\documentclass [12pt]{article}
\usepackage{amssymb,amsmath,comment,amsthm}

\def\N{\mathbb{N}}

\def\Z{\mathbb{Z}}

\def\F{\mathbb{F}}

\newtheorem{theorem}{Theorem}[section]
\newtheorem{proposition}[theorem]{Proposition}
\newtheorem{corollary}[theorem]{Corollary}
\newtheorem{lemma}[theorem]{Lemma}

\newtheorem{conjecture}[theorem]{Conjecture}

\begin{document}
\title{On Collatz Conjecture for binary polynomials}
\author{ Luis H. Gallardo  and Olivier Rahavandrainy\\
Univ. Brest, UMR CNRS 6205\\
Laboratoire de Math\'ematiques de Bretagne Atlantique\\
6, Avenue Le Gorgeu, C.S. 93837, 29238 Brest Cedex 3, France.\\
e-mail: Luis.Gallardo@univ-brest.fr \\
Olivier.Rahavandrainy@univ-brest.fr}
\maketitle

\begin{itemize}
\item[a)]
Running head: Collatz Conjecture
\item[b)]
\item[c)]
\item[d)]
Corresponding author:
\begin{center} O. Rahavandrainy
\end{center}
\end{itemize}
{\bf{Abstract}}
We build a variant of  Collatz Conjecture for polynomials over $\F_2$ and we prove that it is solved. By the way, we give several examples.
{\section{Introduction}}
The Collatz conjecture is one of the most famous unsolved problems in Arithmetics. As written in (\cite{ref-wiki}), ``it concerns sequences of integers in which each term is obtained from the previous term as follows: if the previous term is even, the next term is one half of the previous term. If the previous term is odd, the next term is 3 times the previous term plus 1. The conjecture says that these sequences always reach 1, no matter which positive integer is chosen to start the sequence''.
We may reformulate this construction. For a given positive integer $n$, consider the $2$-adic valuation $a_0$ of $n$: $n=2^{a_0} n_1$, where $n_1$ is odd.  Put $n_2 = 1+3n_1$, and again, consider the $2$-adic valuation,  $a_2$, of $n_2$:  $n_2=2^{a_2} n_3$, $n_3$ odd  and so on...
We get two sequences of odd and even integers: $[n_1, n_3,\ldots]$ and $[n_2, n_4,\ldots]$. The conjecture states that for any integer $n$, there exists a finite integer $m$ such that for all $t \geq m$, $n_{2t}= 2$ and   $n_{2t+1}= 1$. So, the above two sequences $(n_{2k})_k$ and  $(n_{2k+1})_k$ are both eventually constant.

Many studies are done in order to reach a proof of that Conjecture. See for example, \cite{Furuta} and \cite{Izadi}.\\

Now, we consider a variant (of this problem) with binary polynomials. Let $A \in \F_2[x]$ be a nonzero polynomial. We may think of $x(x+1) \in \F_2[x]$ as being the analogue of $2 \in \Z$. So, we say (\cite{Gall-Rahav4}) that $A$ is \emph{odd} if $\gcd(A,x(x+1)) = 1$, i.e., if it has no linear factor. $A$ is \emph{even}, otherwise. The first odd polynomial after $1$ is $M:=x^2+x+1$. So, the variant of $1+3n$ (for integers) is $1+MA$ (for polynomials).

We denote by $val_x(S)$ (resp. $val_{x+1}(S)$) the valuation at $x$ (resp. at $x+1$) of a polynomial $S$:
$$S=x^{val_x(S)} (x+1)^{val_{x+1}(S)} S_1, \text{ where $S_1$ is odd}.$$

For a fixed nonzero binary polynomial $A$, we define the ``Collatz transformations'' by giving the following three sequences of integers $(a_{2k})_k$ and of polynomials $(A_{2k})_k$, $(A_{2k+1})_k$:\\
\\
$A_0= A,\ a_0 =val_x(A_0)$, $b_0 = val_{x+1}(A_0)$, \\
$A_1$ the odd polynomial such that $A_0 = x^{a_0}(x+1)^{b_0} A_1$,\\
$A_2 = 1+MA_1$, $a_2=val_x(A_2),\ b_2=val_{x+1}(A_2)$, $\displaystyle{A_3 = \frac{A_2}{x^{a_2}(x+1)^{b_2}}}$. \\
$\vdots$
\\
$A_{2k} = 1+MA_{2k-1}$, $a_{2k}=val_x(A_{2k}),\ b_{2k}=val_{x+1}(A_{2k})$, $\displaystyle{A_{2k+1} = \frac{A_{2k}}{x^{a_{2k}}(x+1)^{b_{2k}}}}$. \\
$\vdots$
\\
Note that $A_0$ may be odd and $A_{2k}$ (resp. $A_{2k-1}$) is even (resp. odd) if $k \geq 1$.\\
We may formulate a Collatz Conjecture for binary polynomials as follows.

\begin{conjecture} \label{collatzconject}
For a given $A \in \F_2[x] \setminus \{0\} $, there exists $m \in \N^*$ such that for all $k \geq m$, $A_{2k} = x(x+1)$ and $A_{2k+1} =1$.
\end{conjecture}
We prove it by Theorem \ref{collatz0}.

\begin{theorem} \label{collatz0}
Let $A$ be nonzero binary polynomial, then the sequences of polynomials obtained from Collatz transformations are of finite length ${\ell}_A$.\\
More precisely, these sequences are:
$$[A_0,\ldots, A_{2m-2}, x^2+x] \text{ and } \ [A_1,\ldots, A_{2m-1},1],$$
where $m \in \N^*$, ${\ell}_A = m +1 \leq 2^{\deg(A)-1}$.
\end{theorem}
The upper bound for the length ${\ell}_A$ seems too big. But, we are unable to improve it. Many computations show that $\displaystyle{{\ell}_A \leq \frac{\deg(A)}{2}}$. However, it happens that ${\ell}_A  \simeq \deg(A)$. For several families of polynomials, we obtain some regularity on the degrees involved in the sequences. We also may have $\displaystyle{{\ell}_A \simeq \frac{\deg(A)}{2}}$ (see Section \ref{exemples}).

\section{Proof of Theorem \ref{collatz0}} \label{theoremproof}
First, we shall need general results about the numbers of odd and even polynomials of a given degree $d \geq 2$.
Denote by ${\mathcal{P}}_d$ the set of all polynomials of degree $d$ and consider:
$$\begin{array}{l}
{\mathcal{P}}^{0,0}_d:=\{S \in {\mathcal{P}}_d : S(0) = 0\}, {\mathcal{P}}^{0,1}_d:=\{S \in {\mathcal{P}}_d : S(0) = 1\}\\
\\
{\mathcal{P}}^{1,0}_d:=\{S \in {\mathcal{P}}_d : S(1) = 0\}, {\mathcal{P}}^{1,1}_d:=\{S \in {\mathcal{P}}_d : S(1) = 1\}\\
\\
{\mathcal{O}}_d:=\{S \in {\mathcal{P}}_d : S \text{ is odd}\} =  {\mathcal{P}}^{0,1}_d \cap {\mathcal{P}}^{1,1}_d.
\end{array}$$
One has ${\mathcal{P}}_0 = {\mathcal{O}}_0 =\{1\}$ and ${\mathcal{O}}_1=\emptyset$.
\begin{lemma} \label{oddpolyn}
i) The four sets ${\mathcal{P}}^{0,0}_d, {\mathcal{P}}^{0,1}_d, {\mathcal{P}}^{1,0}_d$ and ${\mathcal{P}}^{1,1}_d$ have all the same cardinality $2^{d-1}$.\\
ii) The set ${\mathcal{O}}_d$ contains exactly $2^{d-2}$ polynomials if $d \geq 2$.
\end{lemma}
\begin{proof}
i): By the bijective map: $S \mapsto S+1$, one has $\# {\mathcal{P}}^{0,0}_d= \#{\mathcal{P}}^{0,1}_d$ and $\# {\mathcal{P}}^{1,0}_d= \#{\mathcal{P}}^{1,1}_d$.
Analogously, the bijection: $S(x) \mapsto S(x+1)$ gives $\# {\mathcal{P}}^{0,0}_d= \#{\mathcal{P}}^{1,0}_d$. It remains then to see that  ${\mathcal{P}}_d$ is a disjoint union of
${\mathcal{P}}^{0,0}_d$ and ${\mathcal{P}}^{0,1}_d$ and $\#{\mathcal{P}}_d = 2^d$.\\
ii): By induction on $d$. The case $d=2$ is trivial since ${\mathcal{O}}_2 = \{x^2+x+1\}$. Now, suppose that $\#{\mathcal{O}}_s = 2^{s-2}$, for $2 \leq s \leq d-1$. We remark that
${\mathcal{P}}^{0,1}_d =\{(x+1)^s S_1: 0 \leq s \leq d, s \not= d-1, S_1 \in {\mathcal{O}}_{d-s} \}$. Hence,
$$\# {\mathcal{P}}^{0,1}_d = \# {\mathcal{O}}_{d} +  \# {\mathcal{O}}_{d-1}+\cdots+ \#  {\mathcal{O}}_{2}+\#  {\mathcal{O}}_{0}.$$
Therefore, $2^{d-1} = \# {\mathcal{O}}_{d} + 2^{d-3} +\cdots + 1 + 1 =  \# {\mathcal{O}}_{d} + 2^{d-2}$ and $\# {\mathcal{O}}_{d} = 2^{d-2}$.
\end{proof}
For $k \geq 0$, we put $d_{2k}:=\deg(A_{2k})$ and $d_{2k+1}:=\deg(A_{2k+1})$.\\
We obviously get the following lemmas.
\begin{lemma} \label{lesakbk}
One has $a_0, b_0 \geq 0$ and $a_{2k}, b_{2k} \geq 1$, for any $k \geq 1$.
\end{lemma}
\begin{proof} If $k \geq 1$, then $x$ and $x+1$ both divide $A_{2k}$, because for $t \in \{0,1\}$, $A_{2k}(t) = 1+M(t) A_{2k-1}(t) = 1 + 1 = 0$.
\end{proof}
\begin{lemma} \label{lesdkellk} If $k \geq 1$, then
$$d_{2k+1} \leq d_{2k-1} \leq \deg(A),\ d_{2k} = d_{2k-1} + 2,\ d_{2k}=d_{2k+1} + a_{2k} + b_{2k}.$$
\end{lemma}~\\
Since $(d_{2k+1})_k$ is a non-negative and non-increasing sequence, we obtain the
\begin{corollary} \label{ellCV}
The sequence $(d_{2k+1})_k$ and $(d_{2k})_k$ are both convergent.
One has~: $$\text{$\lim_k d_{2k+1} = p_1, \ \lim_k d_{2k} = p_2$ where $p_2= p_1 + 2$.}$$
\end{corollary}
\begin{corollary} \label{stationnaire}
There exists $m \geq 1$ such that for any $k \geq m$:
$$d_{2k+1} = p_1,\ d_{2k} = p_2,\ a_{2k} = b_{2k} = 1.$$
\end{corollary}
\begin{proof}
The convergent sequence $(d_{2k+1})_k$ takes its values in the finite set $\{0,1,\ldots, \deg(A)\}$. So, it is eventually constant.
\end{proof}
\begin{corollary} \label{samedegrees}
For any $k \geq m$, the polynomials $A_{2k}$ and $A_{2k+1}$ are respectively of degree $p_2$ and $p_1$.
\end{corollary}
\begin{corollary} \label{cyclic}
There exists a positive integer $t \leq \deg(A)$ such that the polynomials $A_{2(m+t)} = A_{2m}$ and $A_{2(m+t)+1} =A_{2m+1}$.
\end{corollary}
\begin{proof} For any $k \geq m$, the polynomial $A_{2k}$ (resp. $A_{2k+1}$) lies in the finite set of polynomials of degree $p_2$ (resp. $p_1$).
\end{proof}
\begin{proposition} \label{lesAjodd=1}
For any $k \geq m$, $A_{2k+1} = 1$ so that $p_1 = 0$ and $t = 1$.
\end{proposition}
\begin{proof}
For $k \geq m$, $a_{2k} = b_{2k} = 1$, so the Collatz transformations give:
$$\left\{\begin{array}{l}
M A_{2m+1} + (1+M) A_{2m+3} = 1\\
M A_{2m+3} + (1+M) A_{2m+5} = 1\\
\vdots\\
M A_{2m+2t-3} + (1+M) A_{2m+2t-1} = 1\\
M A_{2m+2t-1} + (1+M) A_{2m+2t+1} = 1
\end{array} \right.$$
Since $A_{2m+2t+1} = A_{2m+1}$, we get a linear system of $t$ equations with coefficients in $\F_2[x]$ and
$t$ unknowns: $A_{2m+1},\ldots, A_{2m+2t-1}$. Its matrix $C$ is circulant with first line: $[M, 1+M, 0,\ldots, 0]$. The second member is the transpose of $[1 \ldots 1]$.

By expanding along the first column of $C$, we see that $$\det(C) = {M}^t + (1+M)^t,$$ which is nonzero.

Thus, this system admits a unique $t$-tuple solution which is $(1,\ldots, 1)$.
\end{proof}
\begin{corollary} \label{casell=0}
The even and odd sequences are respectively:
$$[A_2,\ldots, A_{2m-2}, x^2+x], \ [A_1,\ldots, A_{2m-1},1].$$
Moreover, they are of length $m+1 \leq 2^{\deg(A)-1}$.
\end{corollary}
\begin{proof}
We have just seen that $p_1 = 0$ and $t=1$. So, $p_2 = 2$, $A_{2m+1} = 1$ and $A_{2m} = x^2+x$.
The odd sequence contains at most:
$$\begin{array}{l}
\text{- all odd polynomials of degree $\deg(A)$}, \\
\text{- all odd polynomials of degree $\deg(A)-1$}\\
\vdots\\
\text{- the polynomials $x^2+x+1$ and $1$.}
\end{array}$$
Thus, by Lemma \ref{oddpolyn}-ii), one has $$m+1 \leq 2^{\deg(A)-2}+2^{\deg(A)-3} +\cdots+2+1+1 =2^{\deg(A)-1}.$$
\end{proof}
\section{Examples and ``Conceivable'' facts} \label{exemples}
In this section, we determine the lengths of Collatz (odd polynomials) sequences for several families. In each example, we only give the sequence of their degrees. We recall that $M:=x^2+x+1$ (the first odd and non-constant polynomial).
\subsection{Family $\{M^{2^r}+ \cdots+ M+1: r \geq 1 \}$} \label{cas0}
\begin{lemma} \label{lemmecas1j=0}
If $A = M^{2^r}+ \cdots+ M+1$ with $r \geq 1$, then for any $0 \leq k \leq 2^r-2$, $A_{2k+1} = 1+M^{k+1}(M+1)^{2^r-k-1}$, $\deg(A_{2k+1}) = \deg(A)$ and $A_{2(2^r-1)+1} = 1$. The length ${\ell}_A$ equals $2^r$.
\end{lemma}
\begin{proof}
First, for $0 \leq k \leq 2^r-2$, $A_{2k+1}$ is odd and $\deg(A_{2k+1}) = \deg(A)$. We proceed by induction on $k$. If $k=0$, then $A_1 = A$ because $A$ is odd. We easily see that $A= 1+M(M+1)^{2^r-1}$. Suppose that $A_{2k+1} = 1+M^{k+1}(M+1)^{2^r-k-1}$ and prove that $A_{2k+3} = 1+M^{k+2}(M+1)^{2^r-k-2}$, for $k+1 \leq 2^r-2$. We get
$A_{2k+2} = 1+MA_{2k+1} = \cdots= (1+M)[1+M^{k+2}(M+1)^{2^r-k-2}]$ with $2^r-k-2 \geq 1$. So, $1+M^{k+2}(M+1)^{2^r-k-2}$ is odd and $A_{2k+3} = 1+M^{k+2}(M+1)^{2^r-k-2}$.\\
Now, for $k=2^r-2=m$, one has $A_{2m+2} = 1+MA_{2m+1} = \cdots =(1+M)^{2^r}$ and $A_{2m+3}=1$. So, the length ${\ell}_A$ of $[A_1, \ldots, A_{2(2^r-2)+1}, 1]$ equals $m+1 + 1 = (2^r-1) + 1 = 2^r$.
\end{proof}
\subsection{Family $\{(M^{2^{r}}+ \cdots+ M+1)^{2^u}: r \geq 1, u \geq 1\}$} \label{casM2r}
\begin{lemma} \label{lemmecasM2r}
If $A=(M^{2^{r}}+ \cdots+ M+1)^{2^u}$, then ${\ell_A} = 2^u \cdot (2^r-1) + 1$.
\end{lemma}
\begin{proof}
As above, $A_1 = A$. One has
$$A_2 = 1+M + M^{2^u + 1}(1+M+\cdots+ M^{2^r-1})^{2^u} = (1+M) + M^{2^u+1} \cdot (1+M)^{2^u(2^r-1)}.$$ Thus, $A_3  = 1+ M^{2^u+1} \cdot (1+M)^{2^u(2^r-1)-1}$ and $\deg(A_3) = \deg(A_1) = \deg(A)$. We see (by induction on $k$) that $$\text{$A_{2k+1} = 1+ M^{2^u+k} \cdot (1+M)^{2^u(2^r-1)-k}$ and $\deg(A_{2k+1}) = \deg(A)$.}$$ For $k = 2^u(2^r-1)-1 = m$, we get $A_{2m+1} =  1+ M^{2^u+m} \cdot (1+M)$,
$$A_{2m+2} = (1+M)+M^{2^u+m+1} \cdot (1+M) = (1+M) (1+M^{2^u+m+1}) = (1+M) \cdot  (1+M)^{2^{u+r}}.$$ So, $A_{2m+3} = 1$ and ${\ell_A} = m+1 + r_{A_{2m+3}} =m+1+ 1 = 2^u \cdot (2^r-1)+1$.
\end{proof}
\subsection{Family $\{(M^{2^{r}-2v}+ \cdots+ M+1)^{2^u}: r, v \geq 1, u \geq 1\}$} \label{casM2r-2v}
We assume that $2^r-2v$ is not a power of $2$ (this case is already treated in Section \ref{casM2r}).
\begin{lemma} \label{lemmecasM2r-2}
For $r \geq 2$, $u\geq 1$ and $A=(M^{2^{r}-2}+ \cdots+ M+1)^{2^u}$, the length ${\ell_A}$ equals $2^u + 1$.
\end{lemma}
\begin{proof}
$A_1 = A$ since $A$ is odd. $A_2= 1+M + M^{2^u+1} \cdot (1+M+\cdots+M^{2^r-3})^{2^u}$.
Therefore,
$$\begin{array}{lcl}
A_2& =& 1+M + M^{2^u+1} \cdot (1+M)^{2^u}(1+M+\cdots+M^{2^{r-1}-2})^{2^{u+1}}\\
&=& (1+M) \cdot [1+M^{2^u+1} \cdot (1+M)^{2^u-1}(1+M+\cdots+M^{2^{r-1}-2})^{2^{u+1}}].
\end{array}$$
Thus, $A_3 = 1+M^{2^u+1} \cdot (1+M)^{2^u-1}(1+M+\cdots+M^{2^{r-1}-2})^{2^{u+1}}$.\\
We see (by induction on $k$) that
$$A_{2k+1} = 1+M^{2^u+k} \cdot (1+M)^{2^u-k}(1+M+\cdots+M^{2^{r-1}-2})^{2^{u+1}}.$$
In particular, for $k =2^u-1 = m$, one has: $$A_{2m+1} = 1+M^{2^{u+1}-1} \cdot (1+M)(1+M+\cdots+M^{2^{r-1}-2})^{2^{u+1}}.$$
So, $$\begin{array}{lcl}
A_{2m+2}& = &1+MA_{2m+1} = (1+M)(1+M^{2^{u+1}}(1+M+\cdots+M^{2^{r-1}-2})^{2^{u+1}}\\
&= &(1+M)(1+M+\cdots+M^{2^{r-1}-1})^{2^{u+1}}\\
& =& (1+M)((1+M)^{2^{r-1}-1})^{2^{u+1}}.
\end{array}$$
We deduce that $A_{2m+3} = 1$ and ${\ell_A} = m+1 + {\ell}_{A_{2m+3}} = m+1+1 = 2^u+1$.
\end{proof}
\begin{proposition} \label{lemmecasM2r-2v}
If $u,v \geq 1, r\geq 2$ and $A=(M^{2^{r}-2v}+ \cdots+ M+1)^{2^u}$, then ${\ell_A} = 2^u(2v-1) + 1$.
\end{proposition}
\begin{proof}
The case where $v=1$ is already treated above. Suppose that $v \geq 2$. Put $2v = 2^{t_1}s_1$ with $s_1$ odd.
$A_1 = A$ because $A$ is odd.
$$\begin{array}{lcl}
A_2 &= &1+M + M^{2^u+1} \cdot (1+M)^{2^u(2^t-1)}(1+M+\cdots+M^{2^{r-t_1}-s_1+1})^{2^{u+t_1}}\\
&= &(1+M)\cdot [1+  M^{2^u+1} \cdot (1+M)^{2^u(2^{t_1}-1)-1}(1+M+\cdots+M^{2^{r-t_1}-s_1+1})^{2^{u+t_1}}].
\end{array}$$
So, $A_3 =1+  M^{2^u+1} \cdot (1+M)^{2^u(2^{t_1}-1)-1}(1+M+\cdots+M^{2^{r-t_1}-s_1+1})^{2^{u+t_1}}$.\\
We see (by induction on $k$) that
$$A_{2k+1} = 1+  M^{2^u+k} \cdot (1+M)^{2^u(2^{t_1}-1)-k}(1+M+\cdots+M^{2^{r-t_1}-s_1+1})^{2^{u+t_1}}.$$ In particular, for $k =2^u(2^{t_1}-1)-1 = m_1$, one has: $$A_{2m_1+1} = 1+M^{2^{u}+m_1} \cdot (1+M) (1+M+\cdots+M^{2^{r-t_1}-s_1+1})^{2^{u+t_1}}.$$
$\bullet$ If $s_1 = 1$, then $$A_{2m_1+2} = \cdots= (1+M) \cdot (1+M+\cdots+M^{2^{r-t_1}-1})^{2^{u+t_1}} = (1+M)^{2^{u+t_1}(2^{r-t_1}-1)+1}.$$
Thus, $A_{2m_1+3}=1$ and ${\ell_A} = m_1+1+1 = 2^u(2^{t_1}-1)+1$.\\
$\bullet$ If $s_1 \geq 3$, then set $s_1-1 = 2^{t_2}s_2$, $s_2$ odd, $s_2 < s_1$. \\
In this case,
$$A_{2m_1+2} = \cdots= (1+M)^{2^{t_2}} \cdot (1+M+\cdots+M^{2^{r-t_1-t_2}-s_2-1})^{2^{u+t_1+t_2}}.$$
Hence, $A_{2m_1+3} = (1+M+\cdots+M^{2^{r-t_1-t_2}-s_2-1})^{2^{u+t_1+t_2}}$ and
$${\ell_A} = m_1+1 + {\ell}_{A_{2m_1+3}} = 2^u(2^{t_1}-1)+r_{A_{2m_1+3}}.$$
We remark that $A_{2m_1+3}$ has the same form as $A_1$, with $r-t_1-t_2$ instead of $r$,  $s_2+1$ instead of $2v$ and $u+t_1+t_2$ instead of $u$.\\
- If $s_2 = 1$, then $s_1-1 = 2^{t_2}$ and by Lemma \ref{lemmecasM2r-2}, ${\ell}_{A_{2m_1+3}} = 2^{u+t_1+t_2}+1$,
$${\ell_A} = m_1+1+ r_{A_{2m_1+3}} = 2^u(2^{t_1}-1)+ 2^{u+t_1+t_2}+1 = \cdots = 2^u(2v-1)+1.$$
- If $s_2 \geq 3$, then by putting $s_2+1 = 2^{t_3}s_3$, $s_3$ odd and $m_2 = 2^{u+t_1+t_2}(2^{t_3}-1)-1$, one has ${\ell}_{A_{2m_1+3}} = m_2 + 1 + {\ell}_{A_{2m_2+3}}$, $s_3 < s_2 < s_1$.\\
\\
And so on...\\
\\
We obtain the following natural number sequences:
$$\left\{\begin{array}{l}
\text{$\cdots <s_3 < s_2 < s_1$ and $t_1, t_2, t_3,\ldots$
where $s_1, s_2, s_3,\ldots$ are all odd,}\\
\text{$2v = 2^{t_1}s_1,\ s_1 -1 =2^{t_2}s_2$ if $s_1 \geq 3$},\\
\text{$s_2 + 1 = 2^{t_3}s_3$ if $s_2 \geq 3$}\\
\vdots\\
\end{array}
\right.$$
Therefore, there exists $c \in \N^*$ such that $s_{2c-1} = 1$ or $s_{2c} = 1$.\\
$\star$ If $s_{2c-1} = 1$, then
\begin{equation} \label{equa-s-et-t-1}
\left\{\begin{array}{l}
s_1 -1 =2^{t_2}s_2,\\
s_2+1 = 2^{t_3}s_3,\\
\vdots\\
s_{2c-3} -1 =2^{t_{2c-2}}s_{2c-2},\\
s_{2c-2} +1 = 2^{t_{2c-1}}.
\end{array}
\right.
\end{equation}
We need the following notations.
$$\begin{array}{l}
B^1_1= A_1, \ldots, B^1_{2m_1+1}=A_{2m_1+1},\ m_1 = 2^u(2^{t_1}-1)-1,\\
B^3_1= A_{2m_1+3}, \ldots, B^3_{2m_3+1} = (B^3_1)_{2m_3+1},\ m_3 = 2^{u+t_1+t_2}(2^{t_3}-1)-1,\\
\vdots\\
B^{2c-1}_1= {(B^{2c-3}_1)}_{2m_{2c-3}+3}, m_{2c-1} = 2^{u+t_1+t_2+\cdots+t_{2c-3}+t_{2c-2}}(2^{t_{2c-1}}-1)-1.
\end{array}$$
The odd Collatz polynomial sequence for $A$ is the union of:\\
$[B^1_1,..., B^1_{2m_1+1}], \ [B^3_1,..., B^3_{2m_3+1}], ..., 
[B^{2c-1}_1,..., B^{2c-1}_{2m_{2c-1}+1}], \ [B^{2c-1}_{2m_{2c-1}+3}=1],$
which respectively are of length: $m_1+1, m_3+1, \ldots, m_{2c-1}+1$ et $1$.\\
So, we get $${\ell_A} = 2^u(2^{t_1}-1) +2^{u+t_1+t_2}(2^{t_3}-1)+\cdots + 2^{u+t_1+\cdots+t_{2c-3}+t_{2c-2}}(2^{t_{2c-1}}-1)+1.$$
By means of relations in (\ref{equa-s-et-t-1}), we see that ${\ell_A} = 2^u(2v-1)+1$.\\
$\star$ If $s_{2c} = 1$, then
\begin{equation} \label{equa-s-et-t-2}
\left\{\begin{array}{l}
s_1 -1 =2^{t_2}s_2,\\
s_2+1 = 2^{t_3}s_3,\\
\vdots\\
s_{2c-3} -1 =2^{t_{2c-2}}s_{2c-2},\  s_{2c-2} +1 = 2^{t_{2c-1}s_{2c-1}},\\
s_{2c-1}-1 = 2^{t_{2c}}
\end{array}
\right.
\end{equation}
The odd Collatz polynomial sequence for $A$ is the union of:
$$[B^1_1,\ldots, B^1_{2m_1+1}], \ [B^3_1,\ldots, B^3_{2m_3+1}], \ldots, [B^{2c-1}_1,\ldots, B^{2c-1}_{2m_{2c-1}+1}]$$
with the sequence for $B := (1+M+\cdots + M^{2^{a}-2})^{2^{b}}$,
where $$\mbox{$a = r-t_1-t_2 -\cdots-t_{2c-1}-t_{2c}$ and $b=u+t_1+t_2 +\cdots+t_{2c-1}+t_{2c}$}.$$
Thus, the length ${\ell}_{B}$ of $B$ equals $2^b+1$ (Lemma \ref{lemmecasM2r-2}) and 
$${\ell_A} = 2^u(2^{t_1}-1) +2^{u+t_1+t_2}(2^{t_3}-1)+\cdots + 2^{u+t_1+\cdots+t_{2c-3}+t_{2c-2}}(2^{t_{2c-1}}-1)+2^b+1.$$
By means of relations in (\ref{equa-s-et-t-2}), we obtain ${\ell_A} = 2^u(2v-1)+1$.
\end{proof}
\begin{corollary} \label{corollM2v}
For $A=(M^{2v}+ \cdots+ M+1)^{2^u}$, ${\ell_A}$ equals $2^u(2^r-2v-1) + 1$, where $r$ is the least integer such that  $2v < 2^r$.
\end{corollary}
\begin{proof}
Apply the above proposition by writting $2v = 2^r -(2^r - 2v)$.
\end{proof}
\subsection{Family $\{M^{2^r-j}+ \cdots+ M+1: r \geq 1, \ 1 \leq j \leq 2^{r-1}-1\}$} \label{cas1}
We suppose that $2^r-j$ is not a power of $2$ (see Section \ref{cas0} for this case).
\begin{lemma}\label{cas2rmoinsj}
i) If $A = M^{2^r-1}+ \cdots+ M+1$, then $A_1 = 1$ and ${\ell_A}=1$.\\
ii) If $A = M^{2^r-2k}+ \cdots+ M+1$ with $r > k \geq 1$, then ${\ell_A}=2k$.\\
iii) If $A = M^{2^r-2k-1}+ \cdots+ M+1$ with $r > 2k+1 \geq 3$, then ${\ell_A} = 2k+1$.
\end{lemma}
\begin{proof}
i): For $j=1$, $A_0 =A = (1+M)^{2^r-1} =x^{2^r-1}(x+1)^{2^r-1} \cdot 1$. \\
ii): For $j=2$, one has $A_1 = A$ because $A$ is odd.
$$A_2 = 1+MA_1 =1+M^{2^r-1}+ \cdots+ M^2+M = (1+M)^{2^r-1}.$$ So, $A_3 = 1$ and ${\ell_A} = 2$.\\
If $j = 2k \geq 4$, then $A_1 = A$ as above. One has:
$$A_2 = 1+MA_1 = M^{2^r-2k+1} +\cdots + M+1 = \frac{M^{2^r-(2k-2)} + 1}{M+1}.$$
Put $2k-2 = 2^uw$ where $u \geq 1$ and $w=2t-1$ is odd. We get
$$A_2 = (M+1)^{2^u-1} \cdot (M^{2^{r-u}-2t} +\cdots + M+1)^{2^u},$$
and thus $A_3 = (M^{2^{r-u}-2t} +\cdots + M+1)^{2^u}$. Proposition \ref{lemmecasM2r-2v} implies that ${\ell}_{A_3} = 2^u \cdot (2t-1) + 1=2k-1$. Hence, ${\ell_A} = {\ell}_{A_3} + 1 = 2k$.\\
iii): $A$ is even. Put $2k = 2^uw$ with $u \geq 1$ and $w = 2t-1$ odd. One has:
$$A = (M+1)^{2^u-1} \cdot (M^{2^{r-u}-2t} +\cdots + M+1)^{2^u}.$$
So, $A_1 = (M^{2^{r-u}-2t} +\cdots + M+1)^{2^u}$, ${\ell_A} = {\ell}_{A_1} = 2^u \cdot (2t-1) + 1 = 2k+1$.
\end{proof}
\subsection{Family $\{M^n+1: n \geq 2\}$} \label{exemple1}
In this section, we take $A :=M^n+1= (x^2+x+1)^n +1$, for $n \geq 2$ so that $A$ is even. Put $n = 2^r u$, where $r \geq 0$ and $u$ odd. One has $A = (M_1+1)^{2^r} \cdot (M^{u-1}+\cdots +M+1)^{2^r}$. Hence, the first polynomial in the odd sequence is $A_1 = (M^{u-1}+\cdots +M+1)^{2^r}$.\\
On the other hand, if $n \geq 2$, then there exists a unique positive integer $r$ such that $2^{r-1} < n \leq 2^r$. Thus, we may write $n = 2^r - j$, with $0 \leq j \leq 2^{r-1} - 1$.
\begin{proposition} \label{partcase1}
Let $A = {M}^{2^r-j} +1$ where $r \geq 1$ and $0\leq j \leq 2^{r-1}-1$. Then, the odd sequence of $A$ is of length $j+1$ (which is small enough).
\end{proposition}
\begin{proof}
- If $j=0$, then $A = M^{2^r}+1 = (M+1)^{2^r}$. So, $A_1 = 1$ and ${\ell_A} = 1$.\\
- If $j=1$, then $A = M^{2^r-1}+1 = (M+1)(M^{2^r-2}+\cdots+M+1)$. Therefore, $A_1 = M^{2^r-2}+\cdots+M+1$ and ${\ell_A} = {\ell}_{A_1} = 2$, by Lemma \ref{cas2rmoinsj}.\\
- If $j = 2k$ with $k = 2t-1$ odd, then
$$A = \cdots= (M+1)^2 \cdot (M^{2^{r-1}-2t} + \cdots +M+1)^2.$$ Thus,
$A_1 = (M^{2^{r-1}-2t} + \cdots +M+1)^2$ and from Proposition \ref{lemmecasM2r-2v}, $${\ell_A} = {\ell}_{A_1} = 2(2t-1)+1= 2k+1.$$
- If $j = 2k$ with $k = 2^sw$ even, $t \geq 1$ and $w=2t-1$ odd, then
$$A = \cdots= (M+1)^{2^{s+1}-2} \cdot (M^{2^{r-s-1}-2t} + \cdots +M+1)^{2^{s+1}}.$$ So,
$A_1 = (M^{2^{r-s-1}-2t} + \cdots +M+1)^{2^{s+1}}$ and from Proposition \ref{lemmecasM2r-2v}, ${\ell_A} = {\ell}_{A_1} =2^{s+1} \cdot (2t-1) + 1= 2k+1$.\\
- If $j = 2k-1$ is odd, then $A = (M+1) \cdot (M^{2^r-2k}+\cdots+M+1)$, $A_1 = M^{2^r-2k}+\cdots+M+1$. By Lemma \ref{cas2rmoinsj}, ${\ell_A} = {\ell}_{A_1} = 2k$.
\end{proof}
For illustration, we give below the odd degree sequences for $T_n = {M}^n+1$,  $n \in \{9,\ldots,16\}$ so that $n = 2^4 -j$, $0\leq j \leq 7$.  Here, the lengths are all smaller than $n=\deg(T_n)/2$.
$$\begin{array}{|l|c|c|}
\hline
n&\text{Degree sequence}&\text{Length}\\
\hline
9&[16, 16, 16, 16, 16, 16, 16, 0]&8\\
10&[16, 16, 16, 16, 16, 16, 0]&7\\
11&[20, 16, 16, 16, 16, 0]&6\\
12&[16, 16, 16, 16, 0]&5\\
13&[24, 24, 24, 0]&4\\
14&[24, 24, 0]&3\\
15&[28, 0]&2\\
16&[0]&1\\
\hline
\end{array}$$

\subsection{Family $\{M^n: n\geq 2\}$} \label{famille2}
We may write $n = 2^r-j$ where $r$ is the least positive integer such that  $u \leq 2^r$ and $0\leq j \leq 2^{r-1}-1$. We prove
\begin{proposition} \label{theofam2}
If $A = {M}^{2^r-j}$ with $r \geq 1$ and $0\leq j \leq 2^{r-1}-1$, then the length ${\ell_A}$ equals $j+1$ (resp. $2^r+1$) if $j \not= 0$ (resp. if $j=0$).
\end{proposition}
\begin{proof}
First, $A$ is odd so that $A_1 = A$.\\
$\bullet$ If $j=1$, then $A_2 = 1+M^{2^r} = (1+M)^{2^r}$. So, $A_3 = 1$ and ${\ell_A} = {\ell}_{A_1} = {\ell}_{A_3} + 1 = 2$.\\
$\bullet$ If $j= 2$, then $A_2 = 1+M^{2^r-1} = (1+M) (M^{2^r-2} + \cdots +M+1)$, $A_3 = M^{2^r-2} + \cdots +M+1$ with ${\ell}_{A_3} = 2$ (Lemma \ref{cas2rmoinsj}). Thus, ${\ell_A} = 3$.\\
$\bullet$ If $j = 2k \geq 4$, then
$A_2 = (1+M) A_3$, where $A_3 = M^{2^r-2k} + \cdots + M+1$, ${\ell}_{A_3} = 2k$ (Lemma \ref{cas2rmoinsj}) and ${\ell_A} = 2k+1$.\\
$\bullet$ If $j=2k+1 = 2^uw+1$ where $u \geq 1$ and $w$ odd, then
$A_2 = (1+M)^{2^u} A_3$ with $A_3 = (M^{2^{r-u}-w-1} + \cdots + M+1)^{2^u}$, ${\ell}_{A_3} = 2^uw+1$ (Proposition  \ref{lemmecasM2r-2v}). Hence, ${\ell_A} =  2^uw+1 +1 = 2k+2$.\\
$\bullet$ Finally, if $j=0$, then $A_2 = 1+M^{2^r+1} = (1+M) A_3$ with $A_3 = M^{2^r} + \cdots + M+1$, ${\ell}_{A_3} = 2^r$ (Lemma \ref{lemmecas1j=0}) so that ${\ell_A} = 2^r + 1$.
\end{proof}
Example for $n = 2^r-j \in \{9,\ldots,16\}$ ($r = 4$ and $0\leq j \leq 7$) :
$$\begin{array}{|l|c|c|}
\hline
n&\text{Degree sequence}&\text{Length}\\
\hline
9&[18, 16, 16, 16, 16, 16, 16, 0]&8\\
10&[20, 20, 16, 16, 16, 16, 0]&7\\
11&[22, 16, 16, 16, 16, 0]&6\\
12&[24, 24, 24, 24, 0]&5\\
13&[26, 24, 24, 0]&4\\
14&[28, 28, 0]&3\\
15&[30, 0]&2\\
16&[32, 32, 32, 32, 32, 32, 32, 32, 32, 32, 32, 32, 32, 32, 32, 32, 0]&17\\
\hline
\end{array}$$

\subsection{Family $\{{(1+M)}^n + 1: n \geq 2\}$} \label{famille3}
We suppose that $n$ is not a power of $2$ (see Section \ref{famille2}, for this case). We may write $n = 2^r-j$ with $r\geq 2$ and  $1\leq j \leq 2^{r-1}-1$.
\begin{proposition} \label{theofam3}
If $A = {(1+M)}^{2^r-j}+1$ where $r \geq 2$ and $1\leq j \leq 2^{r-1}-1$, then the length ${\ell_A}$ equals $2^r+1$.
\end{proposition}
\begin{proof}
For a fixed $j\geq 1$, we prove (by induction on $k$) that
$$\text{$A_{2k+1} = 1+M^k(1+M)^{2^r-j-k}$, $0 \leq k \leq 2^r-j-1$.}$$ If $k=0$, then $A_1 = A = 1+(1+M)^{2^r-j}$ since $A$ is odd.
Suppose that $A_{2k+1} = 1+M^k(1+M)^{2^r-j-k}$. We claim that $A_{2k+3} = 1+M^{k+1}(1+M)^{2^r-j-k-1}$. One has
$A_{2k+2} = 1+MA_{2k+1} = (1+M)(1+M^{k+1}(1+M)^{2^r-j-k-1})$. So, $A_{2k+3} = 1+M^{k+1}(1+M)^{2^r-j-k-1}$.\\
Now, for $k=2^r-j-1=m$, we get $A_{2m+1} = 1+M^{2^r-j-1}(1+M)$ and
$A_{2m+2} = 1+MA_{2m+1} = (1+M) (1+M^{2^r-j})$.\\
$\bullet$ If $j = 2v-1$, then
$$A_{2m+2} = (1+M)(1+M^{2^r-j}) = (1+M)^2 (M^{2^r-2v} + \cdots+M+1).$$ Hence, $A_{2m+3} = M^{2^r-2v} + \cdots+M+1$ and
${\ell}_{A_{2m+3}} = 2v$ (Lemma \ref{cas2rmoinsj}). The odd sequence for $A$ is the union of
$[A_1, A_3, \ldots, A_{2m+1}]$ with the odd sequence for $A_{2m+3}$. Therefore,
${\ell_A} = (m+1)+ 2v= (2^r-2v+1) + 2v=2^r+1$.\\
$\bullet$ If $j = 2^uw$ with $u \geq 1$ and $w$ odd, then
$$A_{2m+2} = (1+M)(1+M^{2^r-j}) = (1+M)^{2^u} (M^{2^{r-u}-w-1} + \cdots+M+1)^{2^u}.$$ So, $A_{2m+3} = (M^{2^{r-u}-w-1} + \cdots+M+1)^{2^u}$ and
${\ell}_{A_{2m+3}} = 2^u \cdot w + 1=j+1$ (Proposition \ref{lemmecasM2r-2v}). The odd sequence for $A$ is the union of
$[A_1, A_3, \ldots, A_{2m+1}]$ with the odd sequence for $A_{2m+3}$. So,
${\ell_A} = (m+1)+ j+1=2^r+1$.
\end{proof}
Example for $n = 2^r-j \in \{9,\ldots,15\}$ ($r = 4$ and $1\leq j \leq 7$) :
$$\begin{array}{|l|c|c|}
\hline
n&\text{Degree sequence}&\text{Length}\\
\hline
9&[18, 18, 18, 18, 18, 18, 18, 18, 18, 16, 16, 16, 16, 16, 16, 16, 0]&17\\
10&[20, 20, 20, 20, 20, 20, 20, 20, 20, 20, 16, 16, 16, 16, 16, 16, 0]&17\\
11&[22, 22, 22, 22, 22, 22, 22, 22, 22, 22, 22, 20, 16, 16, 16, 16, 0]&17\\
12&[24, 24, 24, 24, 24, 24, 24, 24, 24, 24, 24, 24, 16, 16, 16, 16, 0]&17\\
13&[26, 26, 26, 26, 26, 26, 26, 26, 26, 26, 26, 26, 26, 24, 24, 24, 0]&17\\
14&[28, 28, 28, 28, 28, 28, 28, 28, 28, 28, 28, 28, 28, 28, 24, 24, 0]&17\\
15&[30, 30, 30, 30, 30, 30, 30, 30, 30, 30, 30, 30, 30, 30, 30, 28, 0]&17\\
\hline
\end{array}$$

\subsection{Family $\{1+{M}^a(M+1)^b: a,b \geq 2\}$} \label{famille5}
\begin{proposition} \label{theofam5}
For $A = 1+{M}^a(M+1)^b$ where $a,b \geq 2$, we get
$${\ell_A} = \left\{\begin{array}{l}
b+1 \text{ if $a+b=2^r$ with $r \geq 1$}\\
b+2^r(2^w-u)+1 \text{ if } \left\{\begin{array}{l}
\text{$a+b=2^ru$ with $u \geq 3$ odd, $r \geq 1$ and}\\
\text{$w$ is the least positive integer such that $u-1 < 2^w$}
\end{array}
\right.\\
\\
a+2b-1 \text{ if $a+b=2^t+1$ with $t \geq 1$}\\
b+ 2^r-2v \text{ if } \left\{\begin{array}{l}
\text{$a+b=2v+1$, $v$ is not a power of $2$,}\\
\text{$r$ is the least positive integer such that $2v < 2^r$.}
\end{array}
\right.
\end{array}
\right.$$
\end{proposition}
\begin{proof}
One gets $A_1 = A$ since $A$ is odd. We easily see that
$$\begin{array}{l}
\text{$A_{2k+1} = 1+M^{a+k}(1+M)^{b-k}$, for $1 \leq k \leq b-1 = m$,}\\
A_{2m+1} = 1+M^{a+b-1}(1+M), \\
A_{2m+2} = 1+MA_{2m+1} =(1+M)(1+M^{a+b}).
\end{array}$$
$\bullet$ If $a+b = 2^r$, then $A_{2m+2} = (1+M)^{2^r+1}$. So,
$$\text{$A_{2m+3} = 1$ and ${\ell_A} = m+1+1 = b+1$.}$$
$\bullet$ If $a+b = 2^ru$ with $u \geq 3$ odd, then $$A_{2m+2} = (1+M)(1+M^u)^{2^r} = (1+M)^{2^r+1}(1+M+\cdots+M^{u-1})^{2^r}.$$ Thus, $A_{2m+3} = (1+M+\cdots+M^{u-1})^{2^r}$. Corollary \ref{corollM2v} implies that $${\ell_A} = m+1+{\ell}_{A_{2m+3}} = b+ 2^r(2^w-(u-1)-1)+1,$$ $w$ being the least positive integer such that $u-1 < 2^w$.\\
$\bullet$ If $a+b =2^t+1$, then $$A_{2m+2} = (1+M)^2(1+M+\cdots+M^{2^t}), \ A_{2m+3} =1+M+\cdots+M^{2^t}$$ and by Lemma \ref{lemmecas1j=0}, $${\ell_A} = m+1+{\ell}_{A_{2m+3}} = b+ 2^t = a+2b-1.$$
$\bullet$ If $a+b =2v+1$ where $2v$ is not a power of $2$, then
$$A_{2m+2} = (1+M)^2(1+M+\cdots+M^{2v}),\ A_{2m+3} =1+M+\cdots+M^{2v}.$$ Corollary \ref{corollM2v} implies that ${\ell_A} = m+1+{\ell}_{A_{2m+3}} = b+ 2^r-2v$, $r$ being the least positive integer such that $2v < 2^r$.
\end{proof}
Example for $2 \leq a \leq b \leq 5$, $a+b \leq 10$ :
$$\begin{array}{|l|l|l|}
\hline
(a,b)&\text{Degree sequence}&\text{Length}\\
\hline
(2,2)&[8, 8, 0]&3\\
(2,3)&[10, 10, 10, 8, 8, 8, 0]&7\\
(2,4)&[12, 12, 12, 12, 8, 8, 0]&7\\
(2,5)&[14, 14, 14, 14, 14, 12, 0]&7\\
(3,3)&[12, 12, 12, 8, 8, 0]&6\\
(3,4)&[14, 14, 14, 14, 12, 0]&6\\
(3,5)&[16, 16, 16, 16, 16, 0]&6\\
(4,4)&[16, 16, 16, 16, 0]&5\\
(4,5)&[18, 18, 18, 18, 18, 16, 16, 16, 16, 16, 16, 16, 0]&13\\
(5,5)&[20, 20, 20, 20, 20, 16, 16, 16, 16, 16, 16, 0]&12\\
\hline
\end{array}$$
\subsection{Family $\{({M}^2+M+1)^n: n \geq 2\}$} \label{famille4}
We state the following conjecture. Note that Lemma \ref{lemmecasM2r-2} treats the case where $n = 2^u$, $u \geq 1$.
\begin{conjecture} \label{theofam3}
If $A = ({M}^2+M+1)^n$ with $n \geq 2$, then ${\ell}_A$ equals $n+1$.
\end{conjecture}
Example for $n \in \{9,\ldots,16\}$ :
$$\begin{array}{|l|c|c|}
\hline
n&\text{Degree sequence}&\text{Length}\\
\hline
9&[36, 32, 32, 32, 32, 32, 30, 30, 4, 0]&10\\
10&[40, 40, 32, 32, 28, 28, 28, 28, 8, 8, 0]&11\\
11&[44, 42, 36, 34, 32, 28, 26, 26, 12, 10, 4, 0]&12\\
12&[48, 48, 48, 48, 40, 40, 40, 40, 16, 16, 16, 16, 0]&13\\
13&[52, 48, 46, 46, 44, 40, 38, 38, 20, 16, 14, 14, 4, 0]&14\\
14&[56, 56, 52, 52, 48, 48, 44, 44, 24, 24, 20, 20, 8, 8, 0]&15\\
15&[60, 58, 56, 54, 52, 50, 48, 46, 28, 26, 24, 22, 12, 10, 4, 0]&16\\
16&[64, 64, 64, 64, 64, 64, 64, 64, 64, 64, 64, 64, 64, 64, 64, 64, 0]&17\\
\hline
\end{array}$$
\subsection{Family $\{x^n+x+1: n \geq 2\}$}
A priori, this family does not contain any polynomial in $M$, except for $n=2$ and $n=4$.  
We state two conjectures.
\begin{conjecture} \label{proppartcase}
Let $s$ be the greatest integer such that  $n-2^{s+1} \geq 1$. Then, for any positive integer $t \leq s-1$, the odd sequence contains $2^t$ polynomials which have the same degree $d_t$.
In particular, $d_1 = \deg(A_5) = \deg(A_7)$ and  $d_2 = \deg(A_9) = \deg(A_{11}) = \deg(A_{13})=\deg(A_{15})$.
\end{conjecture}

\begin{conjecture} \label{theorempartcase}
Let $s$ be the greatest integer such that  $n-2^{s+1} \geq 1$. Then, the length $\ell_A$ equals $2^s+1$.
\end{conjecture}

Conjecture \ref{theorempartcase} follows from
Conjecture \ref{proppartcase}. Indeed, from Corollary \ref{casell=0}, the sequence of odd polynomials is of length $m+1$:
$$[A_1, A_3, A_5, A_7, \ldots, A_{2m-1-2^s}, \ldots,A_{2m-3},A_{2m-1}, 1].$$
One has, by Conjecture \ref{proppartcase}, $$m+1=1+1+2+2^2+\cdots +2^{s-1} + 1 = 1+(2^s-1) + 1 = 2^s+1.$$
Example for $n \in \{7,8,14,15,16,17,18\}$
$$\begin{array}{|l|c|c|}
\hline
n&\text{Degree sequence}&\text{Length}\\
\hline
7&[7, 5, 0]&3\\
8&[8,4, 0]&3\\
14&[14, 11, 8, 8, 0]&5\\
15&[15, 13, 8, 8, 0]&5\\
16&[16, 12, 8, 8, 0]&5\\
17&[17, 15, 14, 14, 12, 12, 12, 12, 0]&9\\
18&[18, 15, 14, 14, 12, 12, 12, 12, 0]&9\\
\hline
\end{array}$$
\subsection{Remarks} \label{remarques}
We denote by $\overline{S}$ the polynomial obtained from $S \in \F_2[x]$, by replacing $x$ by $x+1$. We also consider the reciprocal  $S^*$ of $S$ as:
$\displaystyle{S^*(x) = x^{\deg(S)} \cdot S(\frac{1}{x}).}$\\
It is easy to see that the Collatz sequences of $\overline{A}$ are exactly obtained from those of $A$ by applying the operation: $S \mapsto \overline{S}$.
But for $A^*$, it is not true (in general). For example, if $A=x^8+x^3+1$, then the odd degree sequence is $[8, 7, 5, 5, 4, 3, 0]$, whereas for $A^*=x^8+x^5+1$, one gets $[8, 6, 6, 0]$.

\end{document}